\documentclass[11pt]{amsart}
\usepackage{cases}

\theoremstyle{plain}
\newtheorem{theorem}                {Theorem}      [section]
\newtheorem*{theorem*}                {Theorem}

\newtheorem{corollary}    [theorem]  {Corollary}
\newtheorem{lemma}        [theorem]  {Lemma}

\theoremstyle{definition}

\newtheorem{remark}       [theorem]  {Remark}
\newtheorem{definition}   [theorem]  {Definition}

\DeclareMathOperator{\trace}{trace} 

\DeclareMathOperator{\Div}{div}

\DeclareMathOperator{\grad}{grad}

\def \final{\mbox{\tiny{$\mathbb{C}P^{p+q+pq}$}}}

\def \Aphi{\mbox{$A^{\phi}$}}
\def \Apsi{\mbox{$A^{\psi}$}}

\def \Bphi{\mbox{$B^{\phi}$}}
\def \Bpsi{\mbox{$B^{\psi}$}}

\def \prod{\mbox{\tiny{$\mathbb{C}P^p\times\mathbb{C}P^q$}}}

\def \prodd{\mbox{$\mathbb{C}P^p\times\mathbb{C}P^q$}}

\numberwithin{equation}{section}

\begin{document}

\title[Segre embedding and biharmonicity]{Segre embedding and biharmonicity}

\author{Hiba~Bibi}

\author{Dorel~Fetcu}

\author{Cezar~Oniciuc}

\thanks{}\dedicatory{Dedicated to Professor Bang--Yen Chen on the occasion of his 80th
birthday}

\address{Institut Denis Poisson, CNRS UMR 7013\\ 
Universit\'e de Tours, Universit\'e d'Orl\'eans\\
Parc de Grandmont, 37200 Tours, France} \email{hiba.bibi@univ-tours.fr}

\address{Department of Mathematics and Informatics\\
Gh. Asachi Technical University of Iasi\\
Bd. Carol I, 11 A\\
700506 Iasi, Romania} \email{dorel.fetcu@academic.tuiasi.ro}

\address{Faculty of Mathematics\\ Al. I. Cuza University of Iasi\\ 
Bd. Carol I, 11\\ 700506 Iasi, Romania} \email{oniciucc@uaic.ro}

\keywords {Biconservative Submanifolds, Biharmonic Submanifolds, Segre Embedding}

\subjclass{32V40, 53C40, 31B30, 53C42}

\begin{abstract} We consider the Segre embedding of the product $\mathbb{C}P^p\times\mathbb{C}P^q$ into $\mathbb{C}P^{p+q+pq}$ and study the biharmonicity of $M^p\times\mathbb{C}P^q$ and $M^p_1\times M^q_2$ as submanifolds of $\mathbb{C}P^{p+q+pq}$, where $M$ and $M_1$ are Lagrangian submanifolds of $\mathbb{C}P^p$ and $M_2$ is a Lagrangian submanifold of $\mathbb{C}P^q$. We find two new large classes of biharmonic submanifolds in complex projective space forms.
\end{abstract}

\maketitle

\section{Introduction}

In mid-$1980$'s, B.-Y.~Chen  \cite{Chen0} defined biharmonic submanifolds of Euclidean spaces $\mathbb{E}^{n}$ as isometric immersions with harmonic mean curvature vector field. In the same period, biharmonic maps between Riemannian manifolds were defined independently by G.-Y.~Jiang \cite{Jiang2009, Jiang}, at a more abstract level, as critical points of the $L^{2}$-norm of the tension field, as previously suggested in $1964$ by J.~Eells and J.~H.~Sampson \cite{ES}. This variational definition coincides with the one proposed by B.-Y.~Chen when the ambient space is $\mathbb{E}^n$ and the map is an isometric immersion. As all harmonic maps (minimal submanifolds in the case of immersions) are biharmonic, the interesting case is that of proper-biharmonic maps (submanifolds), i.e., biharmonic maps (submanifolds) which are not harmonic (minimal). In the same paper \cite{Chen0}, it is conjectured that there are no proper-biharmonic submanifolds in $\mathbb{E}^n$. However, a multitude of examples of proper-biharmonic submanifolds exist in other ambient spaces. The most studied case is that of biharmonic submanifolds in spheres (detailed accounts on these studies can be found  in \cite{Fetcu-Oniciuc-Survey, HT, Chen-Ou}). Then, the next step to be taken has been to study the biharmonic submanifolds in spaces with non-constant sectional curvature, and a very good environment in this respect proved to be complex projective spaces (see, for example, \cite{BCFO, FLMO, I-I-U, S1, S2, Sasahara-2019}).

From the theory of biharmonic submanifolds, a new direction developed in the last decade with the studies on biconservative submanifolds. These submanifolds are defined only by the vanishing of the tangent part of the bitension field. Some examples of articles on this topic are \cite{C-M-O-P, LMO, Manfio-Turgay,  M.O.R., Nistor1, S3, T.U.}.

Although the Segre embedding \cite{Segre} is a notion specific (and an important one) to Algebraic Geometry, it was also often used in studies where the point of view of Differential Geometry prevails (see, for example,  \cite{Chen1, NT, TT} and also \cite{Chen3} for a detailed report on the main results obtained in this approach between mid-$1970$'s and $2002$). However, the Segre embedding, an isometric immersion with a peculiar second fundamental form, was never employed in studying biharmonicity until very recently. In \cite{BCFO}, a paper from $2021$, it is proved that a product $\gamma\times\mathbb{C}P^q(4)$, where $\gamma$ is a curve in $\mathbb{C}P^1(4)$, immersed via Segre embedding in $\mathbb{C}P^{1+2q}(4)$, is proper-biharmonic if and only if $\gamma$ is proper-biharmonic in $\mathbb{C}P^1(4)$. This result suggests that there may be more such examples to be found by using the Segre embedding. 

In our paper, we further exploit this concept in order to find some of these proper-biharmonic submanifolds in complex projective spaces. Thus, we consider the Segre embedding of $\mathbb{C}P^p(4)\times\mathbb{C}P^q(4)$ into $\mathbb{C}P^{p+q+pq}(4)$, and then, by using Lagrangian submanifolds in $\mathbb{C}P^p(4)$ and $\mathbb{C}P^q(4)$, we obtain two classes of proper-biharmonic product submanifolds in $\mathbb{C}P^{p+q+pq}(4)$.

\textbf{Conventions.}
Henceforth, the complex projective space $\mathbb{C}P(4)$ of complex dimension $n$ and constant holomorphic sectional curvature $4$ will be denoted simply by $\mathbb{C}P^{n}$. For curvature tensors we will use the following sign convention
$$
R(X,Y)Z=\nabla_{X}\nabla_{Y}Z-\nabla_{Y}\nabla_{X}Z-\nabla_{[X,Y]}Z.
$$
The Laplacian defined for sections in a Riemannian vector bundle $\pi:E\to M$, endowed with a linear connection $\nabla^E$, will be $\Delta=-\trace(\nabla^E)^2$.

\section{Preliminaries}

Biharmonic maps $\phi : M^{m} \to N^{n}$ between two Riemannian manifolds are critical points of the bienergy functional
$$
E_{2}:C^{\infty}(M,N)\to \mathbb{R}, \quad E_{2}(\phi)=\frac{1}{2}\int_{M} |\tau(\phi)|^{2} dv,
$$
where $\tau(\phi)= \trace  \nabla d \phi$ is the tension field of $\phi$. The Euler-Lagrange equation, also called the biharmonic equation in this case, was derived by G.-Y.~Jiang \cite{Jiang}
\begin{eqnarray}\label{tau-2}
\tau_{2}(\phi)=-\Delta \tau(\phi)-\trace R^N(d\phi(\cdot),\tau(\phi))d \phi (\cdot)=0,
\end{eqnarray}
where $\tau_{2}(\phi)$ is the bitension field of $\phi$.

Any harmonic map is biharmonic and, therefore, we are interested in studying proper-biharmonic maps, i.e., non-harmonic biharmonic maps.

Next, if we consider a fixed map $\phi$ and let the domain metric vary, one obtains a functional on the set $\mathcal{G}$ of
Riemannian metrics on $M$
$$
\mathcal{F}_{2}:\mathcal{G}\to \mathbb{R}, \quad \mathcal{F}_{2}(g)=E_{2}(\phi).
$$
Critical points of this functional are characterized by the vanishing of the stress-energy
tensor $S_{2}$ of the bienergy (see \cite{LMO}). This tensor was introduced
in \cite{Jiang87} as
\begin{eqnarray*}
S_{2}(X,Y)&=&\frac{1}{2}\vert \tau (\phi)\vert ^{2}\langle X,Y \rangle +\langle d \phi, \nabla \tau (\phi) \rangle \langle X, Y \rangle-\langle d\phi (X),\nabla_{Y} \tau (\phi)\rangle
\\
&\ & -\langle d\phi (Y),\nabla_{X} \tau (\phi)\rangle,
\end{eqnarray*}
and it satisfies
$$
\Div S_{2}=\langle \tau_{2}(\phi), d\phi\rangle.
$$

We note that, for isometric immersions, $(\Div S_{2})^{\sharp} =-\tau_{2}(\phi)^{\top}$, where $\tau_{2}(\phi)^{\top}$ is the
tangent part of the bitension field.

\begin{definition}
A submanifold $\phi:M^{m} \to N^{n}$ of a Riemannian manifold $N^{n}$ is called biconservative if $\Div S_{2}=0$.
\end{definition}

As it is easy to see from this definition, a submanifold is biconservative if and only if the tangent part of its bitension field vanishes.

Next, let us recall some basic results in the theory of submanifolds. For a submanifold $M$ in a Riemannian manifold $N$ and any vector fields $X$ and $Y$ tangent to $M$ we have the Gauss Equation
$$
\nabla^N_XY=\nabla_XY+B(X,Y),
$$
where $\nabla$ is the induced connection on $M$ and $B$ is the second fundamental form of the immersion, and also the Weingarten Equation
$$
\nabla^N_XU=-A_UX+\nabla^{\perp}_XU
$$
where $U$ is a normal vector field, $A$ denotes the shape operator of $M$ in $N$ and $\nabla^{\perp}$ is the connection in the normal bundle.

Throughout our paper, we will also use the Gauss Equation of $M$ in $N$
\begin{equation}\label{Gauss2}
\langle R^N(X,Y)Z,W\rangle = \langle R(X, Y)Z,W\rangle +\langle B(X,Z),B(Y,W)\rangle-\langle B(X,W),B(Y,Z)\rangle,
\end{equation}
where $X$, $Y$, $Z$ and $W$ are vector fields tangent to $M$, as well as its Codazzi Equation
\begin{eqnarray}\label{Codazzi}
(\nabla_{X}^{\perp}B)(Y,Z)-(\nabla_{Y}^{\perp}B)(X,Z)=(R^N(X,Y)Z)^{\perp},
\end{eqnarray}
where
$$
(\nabla_{X}^{\perp}B)(Y,Z)=\nabla^{\perp}_{X}B(Y,Z)-B(\nabla_{X}Y,Z)-B(Y,\nabla_{X}Z).
$$

\begin{definition}
Let $M^{m}$ be a submanifold of a Riemannian manifold  $N^{n}$. If the mean curvature vector field $H=(1/m)\trace B$ of $M$ is parallel in the normal bundle, i.e., $\nabla^{\perp} H=0$, then $M$ is called a PMC submanifold. If the mean curvature $|H|$ is constant then $M$ is called a CMC submanifold.
\end{definition}

Now, consider a complex projective space $\mathbb{C}P^n$ with complex structure $J$. The curvature tensor field of $\mathbb{C}P^n$ is given by
\begin{eqnarray}\label{eq:curv}
R^{\mathbb{C}P^n}(X,Y)Z &=&\langle Y,Z\rangle X-\langle X,Z\rangle Y +\langle JY,Z\rangle JX-\langle JX,Z\rangle JY
\\\nonumber && +2\langle JY,X\rangle JZ,
\end{eqnarray}
where $X$, $Y$ and $Z$ are vector fields tangent to $\mathbb{C}P^n$.

\begin{definition}
A submanifold $M^{m}$ of $\mathbb{C}P^n$  with complex structure $J$ is said to be totally real if $JTM^{m}$ lies in the normal bundle of $M^{m}$. If, moreover, the real dimension of $M$ is equal to $n$ then $JTM^n=NM^n$ and $M^n$ is called a Lagrangian submanifold of $\mathbb{C}P^n$.
\end{definition}

We will also need the following result, which characterizes biharmonic and, implicitly, biconservative submanifolds.

\begin{theorem}[\cite{FLMO}]\label{split}
Let $M^m$ be a submanifold of $\mathbb{C}P^n$ such that $JH$ is tangent to $M$. Then
$M$ is biharmonic if and only if
\begin{equation*}
\begin{cases}
-\Delta^{\perp}H-\trace B(\cdot,A_{H}(\cdot))+(m+3)H=0\\
4\trace A_{\nabla^{\perp}_{(\cdot)}H}(\cdot)+m\grad
(|H|^2)=0,
\end{cases}
\end{equation*}
where $\Delta^\perp$ is the Laplacian in the normal bundle of $M$ in $\mathbb{C}P^n$.
\end{theorem}

\begin{corollary}
Let $M^m$ be a submanifold of $\mathbb{C}P^n$ such that $JH$ is tangent to $M$. Then
$M$ is biconservative if and only if
$$
4\trace A_{\nabla^{\perp}_{(\cdot)}H}(\cdot)+m\grad(|H|^2)=0.
$$
\end{corollary}

We end this section with the definition and a basic property of the Segre embedding. Introduced by C.~Segre \cite{Segre} in $1891$, this isometric and holomorphic embedding is given by
$$
S_{pq}:\mathbb{C}P^p\times\mathbb{C}P^q\rightarrow\mathbb{C}P^{p+q+pq}
$$
with
$$
S_{pq}([(z_0,\cdot\cdot\cdot,z_p)],[(w_0,\cdot\cdot\cdot,w_q)])=\left[(z_jw_t)_{0\leq j\leq p,0\leq t\leq q}\right],
$$
where $(z_0,\cdot\cdot\cdot,z_p)$ and $(w_0,\cdot\cdot\cdot,w_q)$ are the homogeneous coordinates in $\mathbb{C}P^p$ and $\mathbb{C}P^q$, respectively. For the sake of simplicity, from now on, we will denote $S_{pq}=j$.

Let $B^j$ be the second fundamental form of the Segre embedding. Since $\mathbb{C}P^p$ and $\mathbb{C}P^q$ are totally geodesic in $\mathbb{C}P^{p+q+pq}$ and also in $\mathbb{C}P^p\times\mathbb{C}P^q$, we have the following property of $B^j$.

\begin{lemma}\label{Bjprop}
If vector fields $X_1$, $X_2$ are tangent to $\mathbb{C}P^p$ and $Y_1$, $Y_2$ to $\mathbb{C}P^q$, then $B^j(X_1,X_2)=B^j(Y_1,Y_2)=0$.
\end{lemma}

\section{Two classes of biharmonic submanifolds}

\subsection{Submanifolds of type $M^p\times\mathbb{C}P^q$}

Let $M^p$ be a Lagrangian submanifold in $\mathbb{C}P^p$ with mean curvature vector field $H$ and denote by $\nabla$, $B$, $A$, and $\nabla^\perp$ the data of this immersion. Consider two more immersions 
$$
i:\Sigma^{p+2q}=M^p\times\mathbb{C}P^q\to\prodd
$$
and
$$
\phi=j\circ i:\Sigma^{p+2q}\to\mathbb{C}P^{p+q+pq},
$$
where $j:\prodd\to\mathbb{C}P^{p+q+pq}$ is the Segre embedding.

From the Gauss equations of $\phi$ and $j$, one obtains
\begin{eqnarray}\label{BBB}
\Bphi(X,Y)&=&\nabla^{\final}_XY-\nabla^{\Sigma}_XY=\nabla^{\prod}_XY-\nabla^{\Sigma}_XY+B^j(X,Y)\\\nonumber&=&B^i(X,Y)+B^j(X,Y),
\end{eqnarray}
for any vector fields $X$ and $Y$ tangent to $\Sigma$.

Let $\{E_a\}_{a=1}^p$ be a local orthonormal frame field on $M^p$ and $\{\bar E_{\alpha}\}_{\alpha=1}^{2q}$ be a local orthonormal frame field on $\mathbb{C}P^q$. 

A very important feature of the immersion $\phi$ is that $\{B^{\phi}(E_a,\bar E_{\alpha})\}$ are orthonormal vector fields (see \cite{Chen1}). This leads, also using $B^i(E_a,\bar E_{\alpha})=0$ and $B^j(JE_a,\bar E_{\alpha}) =JB^j(E_a,\bar E_{\alpha})$, to the following lemma which will be a key ingredient of our computations.

\begin{lemma}\label{Bj} The second fundamental form $B^j$ of the Segre embedding has the following properties:

\begin{enumerate}

\item $\left\{B^{j}(E_a,\bar E_{\alpha})\right\}$ are orthonormal vector fields;

\item $\left\{B^{j}(JE_a,\bar E_{\alpha})\right\}$ are orthonormal vector fields.
\end{enumerate}
\end{lemma}

The mean curvature vector field of $\Sigma$ in $\mathbb{C}P^{p+q+pq}$ is given by
$$
H^\phi=H^i+\frac{1}{p+2q} \left(\sum_{a=1}^p B^j(E_a,E_a)+\sum_{\alpha=1}^{2q} B^j(\bar E_{\alpha},\bar E_{\alpha})\right).
$$
Since $B^j(E_a,E_a)=B^j(\bar E_{\alpha},\bar E_{\alpha})=0$, we have $H^{\phi}=H^i$. Moreover, it is easy to see that $B^i(E_a,E_a)=B(E_a,E_a)$ and $B^i(\bar E_{\alpha},\bar E_{\alpha})=0$, and, therefore, 
\begin{equation}{\label{H}}
H^{\phi}=\frac{p}{p+2q}H.
\end{equation}

For each vector field $E_a$ we have $B^j(E_a,H)=0$ and it follows
\begin{eqnarray*}
\nabla^{\final}_{E_a}H^{\phi}&=&\frac{p}{p+2q}\left(\nabla^{\prod}_{E_a}H+B^j(E_a,H)\right)\\ 
&=&\frac{p}{p+2q}\nabla^{\mathbb{C}P^p}_{E_a}H.
\end{eqnarray*}
This leads to
\begin{equation}\label{nablaperp1}
\nabla_{E_a}^{\perp\phi}H^{\phi}=\frac{p}{p+2q}\nabla^{\perp}_{E_a}H\quad\quad\textnormal{and}\quad\quad
\Aphi_{H^\phi}E_a=\frac{p}{p+2q}A_HE_a.
\end{equation}
In the same way, we get
\begin{equation}\label{nablaperp2}
\nabla_{\bar E_{\alpha}}^{\perp\phi}H^{\phi}=\frac{p}{p+2q}B^j(\bar E_{\alpha},H)\quad\quad\textnormal{and}\quad\quad
\Aphi_{H^\phi}\bar E_{\alpha}=0,
\end{equation}
for any vector field $\bar E_{\alpha}$.

\begin{remark}\label{rem:2} Since $M^p$ is a Lagrangian submanifold in $\mathbb{C}P^p$, its mean curvature vector field can be written as $H=|H|JX$ for some unit vector field $X$ tangent to $M$. From Lemma \ref{Bj} one can easily see that $|B^j(\bar E_{\alpha},H)|=|H||B^j(\bar E_{\alpha},JX)|=|H|$. Thus, the first identity of \eqref{nablaperp2} implies that $\Sigma$ cannot be a PMC submanifold. On the other hand, it is obvious that $\Sigma$ is a CMC submanifold in $\mathbb{C}P^{p+q+pq}$ if and only if $M$ is a CMC submanifold of $\mathbb{C}P^p$.
\end{remark}

\begin{theorem}\label{thm:bi1} If $M^p$ is a Lagrangian submanifold in $\mathbb{C}P^p$, then 

\begin{enumerate}

\item via the Segre embedding of $\mathbb{C}P^{p}\times \mathbb{C}P^{q}$ into
$\mathbb{C}P^{p+q+pq}$, the product $\Sigma^{p+2q}=M^p\times \mathbb{C}P^q$ is a biconservative submanifold of
$\mathbb{C}P^{p+q+pq}$ if and only if $M^p$ is a biconservative submanifold in $\mathbb{C}P^p$;

\item $\Sigma^{p+2q}$ is a proper-biharmonic submanifold in $\mathbb{C}P^{p+q+pq}$ if and only if $M^p$ is a proper-biharmonic submanifold in $\mathbb{C}P^p$.
\end{enumerate}
\end{theorem}

\begin{proof} We note that the vector field $JH^\phi=(p/(p+2q))JH$ is tangent to $M$ and, therefore, to $\Sigma$, which means that we can apply Theorem \ref{split} to write the biconservative equation of the immersion $\phi:\Sigma\to\mathbb{C}P^{p+q+pq}$ as 
\begin{equation}\label{eq:bicons1}
4\trace\Aphi_{\nabla_{(\cdot)}^{\perp\phi} H^\phi}(\cdot)+(p+2q)\grad(|H^\phi|^2)=0.
\end{equation}

In order to compute the first term in the left-hand side of \eqref{eq:bicons1}, consider again the orthonormal frame fields $\{E_a\}_{a=1}^p$ on $M$ and $\{\bar E_{\alpha}\}_{\alpha=1}^{2q}$ on $\mathbb{C}P^q$ and then, from the first equation \eqref{nablaperp1} and Lemma \ref{Bjprop}, we have
\begin{eqnarray*}
\Aphi_{\nabla_{E_a}^{\perp\phi} H^\phi}E_a&=&\frac{p}{p+2q}\Aphi_{\nabla_{E_a}^{\perp} H}E_a\\ 
&=&\frac{p}{p+2q}\left(-\nabla_{E_a}^{\final}\nabla_{E_a}^\perp H+\nabla_{E_a}^{\perp\phi}\nabla_{E_a}^{\perp}H\right)\\ &=&\frac{p}{p+2q}\left(-\nabla_{E_a}^{\prod}\nabla_{E_a}^\perp H-B^j(E_a,\nabla_{E_a}^\perp H)+\nabla_{E_a}^{\perp\phi}\nabla_{E_a}^{\perp}H\right)\\ &=&\frac{p}{p+2q}\left(-\nabla_{E_a}^{\mathbb{C}P^p}\nabla_{E_a}^\perp H+\nabla_{E_a}^{\perp\phi}\nabla_{E_a}^{\perp}H\right),
\end{eqnarray*}
that is
$$
\langle \Aphi_{\nabla_{E_a}^{\perp\phi} H^\phi}E_a,\bar E_{\alpha}\rangle=0,\quad\forall\alpha\in\{1,...,2q\},
$$
and
$$
\langle \Aphi_{\nabla_{E_a}^{\perp\phi} H^\phi}E_a,E_b\rangle=\frac{p}{p+2q}\langle A_{\nabla_{E_a}^{\perp} H}E_a,E_b\rangle,\quad\forall b\in\{1,...,p\}.
$$
Hence
\begin{equation}\label{nablaperpbicons1}
\Aphi_{\nabla_{E_a}^{\perp\phi} H^\phi}E_a=\frac{p}{p+2q} A_{\nabla_{E_a}^{\perp} H}E_a,\quad\forall a\in\{1,...,p\}.
\end{equation}

Next, from \eqref{nablaperp2} and \eqref{BBB}, for any unit vector field $X$ tangent to $\Sigma$ and any $\bar E\alpha$, one obtains
\begin{eqnarray*}
\langle \Aphi_{\nabla_{\bar E_{\alpha}}^{\perp\phi} H^\phi}\bar E_{\alpha},X\rangle&=&\frac{p}{p+2q}\langle \Aphi_{\nabla_{\bar E_{\alpha}}^{\perp\phi} H^\phi}\bar E_{\alpha},X\rangle\\ &=&\frac{p}{p+2q}\langle B^{\phi}(\bar E_{\alpha},X), B^j(\bar E_{\alpha},H)\rangle\\ &=&\frac{p}{p+2q}\langle B^i(\bar E_{\alpha},X)+B^j(\bar E_{\alpha},X), B^j(\bar E_{\alpha},H)\rangle\\ &=&\frac{p}{p+2q}\langle B^j(\bar E_{\alpha},X), B^j(\bar E_{\alpha},H)\rangle.
\end{eqnarray*}
As $B^j(\bar E_{\alpha},\bar E_{\alpha}\rangle=0$, from the Gauss Equation \eqref{Gauss2} of $j$, we have
\begin{eqnarray*}
\langle B^j(\bar E_{\alpha},X), B^j(\bar E_{\alpha},H)\rangle&=&\langle B^j(\bar E_{\alpha},\bar E_{\alpha}), B^j(X,H)\rangle-\langle R^{\final}(\bar E_{\alpha},H)\bar E_{\alpha},X\rangle\\&&+\langle R^{\prod}(\bar E_{\alpha},H)\bar E_{\alpha},X\rangle\\&=&\langle R^{\prod}(\bar E_{\alpha},H)\bar E_{\alpha}- R^{\final}(\bar E_{\alpha},H)\bar E_{\alpha},X\rangle.
\end{eqnarray*}
Since $\bar E_{\alpha}$'s are tangent to $\mathbb{C}P^q$ and $H$ to $\mathbb{C}P^p$, we have, using the definition of the curvature tensor, $R^{\prod}(\bar E_{\alpha},H)\bar E_{\alpha}=0$. From \eqref{eq:curv}, one can also see that $R^{\final}(\bar E_{\alpha},H)\bar E_{\alpha}=-H$ and, therefore, $\langle B^j(\bar E_{\alpha},X), B^j(\bar E_{\alpha},H)\rangle=0$. Thus, one obtains
\begin{equation}\label{nablaperpbicons2}
\Aphi_{\nabla_{\bar E_{\alpha}}^{\perp\phi} H^\phi}\bar E_{\alpha}=0,\quad\forall\alpha\in\{1,...,2q\}.
\end{equation}

Finally, from equation \eqref{H}, we have
\begin{equation}\label{Hbicons}
\grad(|H^{\phi}|^2)=\frac{p^2}{(p+2q)^2}\grad(|H|^2).
\end{equation}

Replacing \eqref{nablaperpbicons1}, \eqref{nablaperpbicons2} and \eqref{Hbicons} in \eqref{eq:bicons1}, the biconservative equation of the immersion $\phi$, readily becomes
$$
4\trace A_{\nabla_{(\cdot)}^{\perp}H}(\cdot)+p\grad(|H|^2)=0,
$$
which, as $JH$ is tangent to $M$, is just the biconservative equation of $M$. 

To prove the second part of the theorem, we will also evaluate the normal part of the biharmonic equation of the immersion $\phi$ as given by Theorem \ref{split}
\begin{equation}\label{bihnormal}
-\Delta^{\perp\phi}H^{\phi}-\trace(B^{\phi}(\cdot,\Aphi_{H^{\phi}}\cdot))+(p+2q+3)H^{\phi}=0.
\end{equation}

With the same notations as before, from \eqref{nablaperp1} and \eqref{nablaperp2}, we have
\begin{eqnarray}\label{eq:7}
-\Delta^{\perp\phi}H^{\phi}&=&\sum_{a=1}^p\left(\nabla^{\perp\phi}_{E_a}\nabla^{\perp\phi}_{E_a}H^{\phi}-\nabla^{\perp\phi}_{\nabla^{\Sigma}_{E_a}E_a}H^{\phi}\right)\\\nonumber&&+\sum_{\alpha=1}^{2q}\left(\nabla^{\perp\phi}_{\bar E_{\alpha}}\nabla^{\perp\phi}_{\bar E_{\alpha}}H^{\phi}-\nabla^{\perp\phi}_{\nabla^{\Sigma}_{\bar E_{\alpha}}\bar E_{\alpha}}H^{\phi}\right)\\\nonumber&=&\frac{p}{p+2q}\sum_{a=1}^p\left(\nabla^{\perp\phi}_{E_a}\nabla^{\perp}_{E_a}H-\nabla^{\perp}_{\nabla_{E_a}E_a}H\right)\\\nonumber&&+\frac{p}{p+2q}\sum_{\alpha=1}^{2q}\left(\nabla^{\perp\phi}_{\bar E_{\alpha}}B^j(\bar E_{\alpha},H)-\nabla^{\perp\phi}_{\nabla^{\mathbb{C}P^q}_{\bar E_{\alpha}}\bar E_{\alpha}}H\right).
\end{eqnarray}

Now, let $X$ be a vector field tangent to $M$ and $V\in C(T\mathbb{C}P^p)$ a normal one. Then
\begin{eqnarray*}
\nabla_X^{\final}V&=&-\Aphi_VX+\nabla^{\perp\phi}_XV\\&=&\nabla_X^{\prod}V+B^j(X,V)=\nabla_X^{\mathbb{C}P^p}V\\&=&-A_VX+\nabla^{\perp}_XV.
\end{eqnarray*}
Since $V$ is tangent to $\mathbb{C}P^p$, we have $\langle \Aphi_VX,Y\rangle=\langle V,B^j(X,Y)\rangle=0$, for
any vector field $Y\in C(T\mathbb{C}P^q)$, which means that $\Aphi_VX$ is tangent to $M$ and the above equality implies
\begin{equation}\label{eq:AV}
\Aphi_VX=A_VX.
\end{equation}

By using equation \eqref{eq:AV} and Lemma \ref{Bjprop}, we get
\begin{eqnarray}\label{eq:8}
\nabla^{\perp\phi}_{E_a}\nabla_{E_a}^{\perp}H&=&\Aphi_{\nabla^{\perp}_{E_a}H}E_a+\nabla^{\final}_{E_a}\nabla_{E_a}^{\perp}H\\\nonumber&=&A_{\nabla^{\perp}_{E_a}H}E_a+\nabla^{\prod}_{E_a}\nabla_{E_a}^{\perp}H+B^j(E_a,\nabla_{E_a}^{\perp}H)\\\nonumber&=&A_{\nabla^{\perp}_{E_a}H}E_a+\nabla^{\mathbb{C}P^p}_{E_a}\nabla_{E_a}^{\perp}H\\\nonumber&=&\nabla_{E_a}^{\perp}\nabla_{E_a}^{\perp}H.
\end{eqnarray}

Next, we have
\begin{equation}\label{eq:9}
\nabla^{\perp\phi}_{\bar E_{\alpha}}B^j(\bar E_{\alpha},H)=\Aphi_{B^j(\bar E_{\alpha},H)}\bar E_{\alpha}+\nabla^{\final}_{\bar E_{\alpha}}B^j(\bar E_{\alpha},H).
\end{equation}
We know that $B^j(\bar E_{\alpha},\bar E_{\alpha})=0$. Then, a simple computation, using \eqref{eq:curv} and the fact that $JH$ is tangent to $M$, gives $R^{\final}(\bar E_{\alpha},H)\bar E_{\alpha}=-H$. All these, together with the Codazzi Equation \eqref{Codazzi} of the Segre embedding $j$, show that $(\nabla^{\perp j}_{\bar E_{\alpha}}B^j)(\bar E_{\alpha},H)=0$ and then $\nabla^{\perp j}_{\bar E_{\alpha}}B^j(\bar E_{\alpha},H)=B^j(\nabla^{\mathbb{C}P^q}_{\bar E_{\alpha}}\bar E_{\alpha},H)$. 

Therefore, one obtains
$$
\nabla^{\final}_{\bar E_{\alpha}}B^j(\bar E_{\alpha},H)=-A^j_{B^j(\bar E_{\alpha},H)}\bar E_{\alpha}+B^j(\nabla^{\mathbb{C}P^q}_{\bar E_{\alpha}}\bar E_{\alpha},H)
$$
and, replacing in \eqref{eq:9}, 
$$
\nabla^{\perp\phi}_{\bar E_{\alpha}}B^j(\bar E_{\alpha},H)=B^j(\nabla^{\mathbb{C}P^q}_{\bar E_{\alpha}}\bar E_{\alpha},H)-(A^j_{B^j(\bar E_{\alpha},H)}\bar E_{\alpha})^{\perp\phi}.
$$

From Lemma \ref{Bj}, since $H=|H|JX$ for some unit vector field $X$ tangent to $M$, it readily follows that $(A^j_{B^j(\bar E_{\alpha},H)}\bar E_{\alpha})^{\perp\phi}=H$ and then
\begin{equation}\label{eq:14}
\nabla^{\perp\phi}_{\bar E_{\alpha}}B^j(\bar E_{\alpha},H)=B^j(\nabla^{\mathbb{C}P^q}_{\bar E_{\alpha}}\bar E_{\alpha},H)-H.
\end{equation}

The second term in the left-hand side of equation \eqref{eq:7} can be written, by using \eqref{BBB}, \eqref{nablaperp1} and \eqref{nablaperp2}, as
\begin{eqnarray}\label{eq:15}
\nonumber\trace B^{\phi}(\cdot,\Aphi_{H^{\phi}}\cdot)&=&\sum_{a=1}^p B^{\phi}(E_a,\Aphi_{H^{\phi}}E_a)\\&=&\frac{p}{p+2q}\sum_{a=1}^p\left(B^i(E_a,A_HE_a)+B^j(E_a,A_HE_a)\right)\\\nonumber&=&\frac{p}{p+2q}\sum_{a=1}^p B^i(E_a,A_HE_a)=\frac{p}{p+2q}\sum_{a=1}^p B(E_a,A_HE_a)\\\nonumber&=&\frac{p}{p+2q}\trace B(\cdot,A_H\cdot).
\end{eqnarray}

We conclude by replacing \eqref{eq:8}, \eqref{eq:14} and \eqref{eq:15} into \eqref{bihnormal} and using \eqref{nablaperp2} to prove that the normal part of the biharmonic equation of $\phi$ is equivalent to
$$
-\Delta^{\perp}H-\trace B(\cdot,A_H\cdot)+(p+3)H=0,
$$
which is the normal part of the biharmonic equation of $M$ in $\mathbb{C}P^p$.
\end{proof}

\subsection{Submanifolds of type $M_1^p\times M_2^q$} Let us consider two Lagrangian submanifolds $M_1^p$ in $\mathbb{C}P^p$ and $M_2^q$ in $\mathbb{C}P^q$. We shall denote by $\nabla$, $B$, $A$, $\nabla^\perp$ and $\bar\nabla$, $\bar B$, $\bar A$, $\bar\nabla^\perp$, respectively, the data of these two immersions. Also, denote by $H_1$ the mean curvature vector field of $M_1$ in $\mathbb{C}P^p$ and by $H_2$ the mean curvature vector field of $M_2$ in $\mathbb{C}P^q$.

Next, consider the direct product $\Sigma^{p+q}=M_1^p\times M_2^q$ immersed, first in $\prodd$, and denote this immersion by $i:\Sigma^{p+q}\to\prodd$, and then in $\mathbb{C}P^{p+q+pq}$, via the Segre embedding $j:\prodd\to\mathbb{C}P^{p+q+pq}$, and denote by $\psi:j\circ i:\Sigma^{p+q}\to\mathbb{C}P^{p+q+pq}$ this second immersion.

In what follows we will study the biharmonicity of $\Sigma^{p+q}$ as a submanifold of $\mathbb{C}P^{p+q+pq}$.

Let $\{E_a\}_{a=1}^p$ be a local orthonormal frame field on $M_1$ and $\{\bar E_{\alpha}\}_{\alpha=1}^{q}$ be a local orthonormal frame field on $M_2$. 

First, since $M_1$ and $M_2$ are Lagrangian submanifolds, we can rewrite Lemma \ref{Bj} in a way adapted to our case as follows.

\begin{lemma}\label{Bj2} The second fundamental form $B^j$ of the Segre embedding has the following properties:

\begin{enumerate}

\item $\left\{B^{j}(E_a,\bar E_{\alpha}),B^{j}(E_b,J\bar E_{\beta})\right\}$ are orthonormal vector fields;

\item $\left\{B^{j}(JE_a,\bar E_{\alpha}),B^{j}(JE_b,J\bar E_{\beta})\right\}$ are orthonormal vector fields;

\end{enumerate}
\end{lemma}

A similar computation to that in the previous case shows that the second fundamental form of the immersion $\psi$ can be written as
\begin{equation}\label{BBB2}
\Bpsi(X,Y)=B^i(X,Y)+B^j(X,Y),
\end{equation}
for any vector fields $X$ and $Y$ tangent to $\Sigma$, in this situation too. Then it is easy to see that the mean curvature vector field of $\psi$ is given by
\begin{equation}\label{H12}
H^\psi=\frac{p}{p+q}H_1+\frac{q}{p+q}H_2.
\end{equation}

Now, for any vector field $X$ tangent to $\Sigma$, we have
\begin{eqnarray*}
\nabla^{\final}_XH^{\psi}&=&\frac{p}{p+q}\left(\nabla^{\prod}_XH_1+B^j(X,H_1)\right)\\&&+\frac{q}{p+q}\left(\nabla^{\prod}_XH_2+B^j(X,H_2)\right).
\end{eqnarray*}
Specializing this formula for $X=E_a$ and then for $X=\bar E_{\alpha}$ and using Lemma \ref{Bjprop}, one obtains
$$
\nabla^{\final}_{E_a}H^{\psi}=\frac{p}{p+q}\nabla^{\mathbb{C}P^p}_{E_a}H_1+\frac{q}{p+q}B^j(E_a,H_2)
$$
and
$$
\nabla^{\final}_{\bar E_{\alpha}}H^{\psi}=\frac{q}{p+q}\nabla^{\mathbb{C}P^q}_{\bar E_{\alpha}}H_2+\frac{p}{p+q}B^j(\bar E_{\alpha},H_1),
$$
for any $a\in\{1,...,p\}$ and $\alpha\in\{1,...,q\}$. From here we easily get
\begin{equation}\label{eq:18}
\nabla_{E_a}^{\perp\psi}H^{\psi}=\frac{p}{p+q}\nabla^{\perp}_{E_a}H_1+\frac{q}{p+q}B^j(E_a,H_2),\quad A^{\psi}_{H^{\psi}}E_a=\frac{p}{p+q}A_{H_1}E_a
\end{equation}
and
\begin{equation}\label{eq:19}
\nabla_{\bar E_{\alpha}}^{\perp\psi}H^{\psi}=\frac{q}{p+q}\bar\nabla^{\perp}_{\bar E_{\alpha}}H_2+\frac{p}{p+q}B^j(\bar E_{\alpha},H_1),\quad A^{\psi}_{H^{\psi}}\bar E_{\alpha}=\frac{q}{p+q}\bar A_{H_2}\bar E_{\alpha}.
\end{equation}

\begin{remark} As in the previous case, Lemma \ref{Bj2} and the first identities of \eqref{eq:18} and \eqref{eq:19} show that $\nabla^{\perp\psi}H^{\psi}$ does not vanish identically and, therefore, $\Sigma$ cannot be a PMC submanifold in $\mathbb{C}P^{p+q+pq}$. Obviously, if $M_1$ and $M_2$ are CMC submanifolds in $\mathbb{C}P^p$ and $\mathbb{C}P^q$, respectively, so is $\Sigma$ in $\mathbb{C}P^{p+q+pq}$, while the converse is not true in general.
\end{remark}

\begin{theorem}\label{thm:bi2} If $M_1^p$ and $M_2^q$ are Lagrangian submanifolds in $\mathbb{C}P^p$ and $\mathbb{C}P^q$, respectively, then 

\begin{enumerate}

\item via the Segre embedding of $\mathbb{C}P^{p}\times \mathbb{C}P^{q}$ into
$\mathbb{C}P^{p+q+pq}$, the product $\Sigma^{p+q}=M_1^p\times M_2^q$ is a biconservative submanifold of
$\mathbb{C}P^{p+q+pq}$ if and only if $M_1^p$ and $M_2^q$ are biconservative submanifolds in $\mathbb{C}P^p$ and $\mathbb{C}P^q$, respectively;

\item $\Sigma^{p+q}$ is a proper-biharmonic submanifold in $\mathbb{C}P^{p+q+pq}$ if and only if one of the submanifolds $M_1^p$ or $M_2^q$ is minimal and the other is proper-biharmonic in $\mathbb{C}P^p$ or $\mathbb{C}P^q$, respectively.
\end{enumerate}
\end{theorem}

\begin{proof} Since $M_1$ and $M_2$ are Lagrangian submanifolds in their initial ambient spaces, we have that 
$$
JH^\psi=\frac{p}{p+q}JH_1+\frac{q}{p+q}JH_2
$$ 
is a vector field tangent to $\Sigma$. This and Theorem \ref{split} imply that the biconservative equation of $\psi:\Sigma\to\mathbb{C}P^{p+q+pq}$ is
\begin{equation}\label{eq:bicons2}
4\trace\Apsi_{\nabla_{(\cdot)}^{\perp\phi} H^\phi}(\cdot)+(p+q)\grad(|H^\psi|^2)=0.
\end{equation}

In the following, we will continue using the orthonormal frame fields $\{E_a\}_{a=1}^p$ on $M_1$ and $\{\bar E_{\alpha}\}_{\alpha=1}^{q}$ on $M_2$. 

Now, the first equation of \eqref{eq:18} and Lemma \ref{Bjprop} give
\begin{eqnarray*}
\Apsi_{\nabla_{E_a}^{\perp\psi} H^\psi}E_a&=&\frac{p}{p+q}\Apsi_{\nabla_{E_a}^{\perp} H_1}E_a+\frac{q}{p+q}\Apsi_{B^j(E_a,H_2)}E_a\\ 
&=&\frac{p}{p+q}\left(-\nabla_{E_a}^{\final}\nabla_{E_a}^\perp H_1+\nabla_{E_a}^{\perp\psi}\nabla_{E_a}^{\perp}H_1\right)\\&&+\frac{q}{p+q}\left(-\nabla^{\final}_{E_a}B^j(E_a,H_2)+\nabla^{\perp\psi}_{E_a}B^j(E_a,H_2)\right) \\&=&\frac{p}{p+q}\left(-\nabla_{E_a}^{\prod}\nabla_{E_a}^\perp H_1-B^j(E_a,\nabla_{E_a}^\perp H_1)+\nabla_{E_a}^{\perp\psi}\nabla_{E_a}^{\perp}H_1\right)\\&&+\frac{q}{p+q}\Apsi_{B^j(E_a,H_2)}E_a \\&=&\frac{p}{p+q}\left(-\nabla_{E_a}^{\mathbb{C}P^p}\nabla_{E_a}^\perp H_1+\nabla_{E_a}^{\perp\phi}\nabla_{E_a}^{\perp}H_1\right)+\frac{q}{p+q}\Apsi_{B^j(E_a,H_2)}E_a.
\end{eqnarray*}
Taking the inner product with any $\bar E_{\alpha}$ one obtains
\begin{eqnarray*}
\langle \Apsi_{\nabla_{E_a}^{\perp\psi} H^\psi}E_a,\bar E_{\alpha}\rangle&=&\frac{q}{p+q}\langle\Apsi_{B^j(E_a,H_2)}E_a,\bar E_{\alpha}\rangle\\&=&\frac{q}{p+q}\langle B^j(E_a,H_2),B^{\psi}(E_a,\bar E_{\alpha})\rangle\\&=&\frac{q}{p+q}\langle B^j(E_a,H_2),B^i(E_a,\bar E_{\alpha})+B^j(E_a,\bar E_{\alpha})\rangle\\&=&\frac{q}{p+q}\langle B^j(E_a,H_2),B^j(E_a,\bar E_{\alpha})\rangle.
\end{eqnarray*}
On the other hand, a simple computation, using \eqref{eq:curv}, shows that 
$$
R^{\final}(E_a,H_2)E_a=-H_2
$$ 
and, therefore, from the Gauss Equation \eqref{Gauss2} of the Segre embedding, it follows that $\langle B^j(E_a,H_2),B^j(E_a,\bar E_{\alpha})\rangle=0$. Thus, $\Apsi_{\nabla_{E_a}^{\perp\psi} H^\psi}E_a$ is tangent to $M_1$. 

Next, we compute
$$
\langle \Apsi_{\nabla_{E_a}^{\perp\psi} H^\psi}E_a,E_b\rangle=\frac{p}{p+q}\langle A_{\nabla_{E_a}^{\perp} H_1}E_a,E_b\rangle,\quad\forall b\in\{1,...,p\},
$$
and conclude with
\begin{equation}\label{eq:23}
\Apsi_{\nabla_{E_a}^{\perp\psi} H^\psi}E_a=\frac{p}{p+q} A_{\nabla_{E_a}^{\perp} H_1}E_a,\quad\forall a\in\{1,...,p\}.
\end{equation}

Working in the same way, we can also prove that
\begin{equation}\label{eq:24}
\Apsi_{\nabla_{\bar E_{\alpha}}^{\perp\psi} H^\psi}\bar E_{\alpha}=\frac{q}{p+q}\bar A_{\bar\nabla_{\bar E_{\alpha}}^{\perp} H_2}\bar E_{\alpha},\quad\forall\alpha\in\{1,...,q\}.
\end{equation}

From equation \eqref{H12}, we easily get
\begin{equation}\label{eq:20}
\grad(|H^{\psi}|^2)=\frac{p^2}{(p+q)^2}\grad(|H_1|^2)+\frac{q^2}{(p+q)^2}\grad(|H_2|^2).
\end{equation}

Replacing \eqref{eq:23}, \eqref{eq:24} and \eqref{eq:20} in \eqref{eq:bicons2}, one obtains that the biconservative equation of the immersion $\psi$ is equivalent to
$$
4\trace A_{\nabla_{(\cdot)}^{\perp}H_1}(\cdot)+p\grad(|H_1|^2)=0\quad\textnormal{and}\quad
4\trace A_{\nabla_{(\cdot)}^{\perp}H_1}(\cdot)+q\grad(|H_2|^2)=0,
$$
as the left-hand side terms are tangent one to $\mathbb{C}P^p$ and the other to $\mathbb{C}P^q$. Since $M_1$ and $M_2$ are Lagrangian submanifolds of $\mathbb{C}P^p$ and $\mathbb{C}P^q$, respectively, these are their biconservative equations and we conclude the first part of the theorem.

The normal part of the biharmonic equation of the immersion $\psi$ is
\begin{equation}\label{eq:bihnormalpsi}
-\Delta^{\perp\psi}H^{\psi}-\trace(B^{\psi}(\cdot,\Apsi_{H^{\psi}}\cdot))+(p+q+3)H^{\psi}=0.
\end{equation}

The first term in \eqref{eq:bihnormalpsi} can be written as 
\begin{eqnarray}\label{eq:25}
-\Delta^{\perp\psi}H^{\psi}&=&\sum_{a=1}^p\left(\nabla^{\perp\psi}_{E_a}\nabla^{\perp\psi}_{E_a}H^{\psi}-\nabla^{\perp\psi}_{\nabla^{\Sigma}_{E_a}E_a}H^{\psi}\right)\\\nonumber&&+\sum_{\alpha=1}^{q}\left(\nabla^{\perp\psi}_{\bar E_{\alpha}}\nabla^{\perp\psi}_{\bar E_{\alpha}}H^{\psi}-\nabla^{\perp\psi}_{\nabla^{\Sigma}_{\bar E_{\alpha}}\bar E_{\alpha}}H^{\psi}\right).
\end{eqnarray}

From the first equation \eqref{eq:18}, we have
\begin{equation}\label{eq:26}
\nabla^{\perp\psi}_{E_a}\nabla^{\perp\psi}_{E_a}H^{\psi}=\frac{p}{p+q}\nabla^{\perp\psi}_{E_a}\nabla^{\perp}_{E_a}H_1+\frac{q}{p+q}\nabla^{\perp\psi}_{E_a}B^j(E_a,H_2).
\end{equation}

Now, using \eqref{eq:AV} and Lemma \ref{Bjprop}, we can compute
\begin{eqnarray}\label{eq:27}
\nabla^{\perp\psi}_{E_a}\nabla_{E_a}^{\perp}H_1&=&\Apsi_{\nabla^{\perp}_{E_a}H_1}E_a+\nabla^{\final}_{E_a}\nabla_{E_a}^{\perp}H_1\\\nonumber&=&A_{\nabla^{\perp}_{E_a}H}E_a+\nabla^{\prod}_{E_a}\nabla_{E_a}^{\perp}H_1+B^j(E_a,\nabla_{E_a}^{\perp}H_1)\\\nonumber&=&A_{\nabla^{\perp}_{E_a}H_1}E_a+\nabla^{\mathbb{C}P^p}_{E_a}\nabla_{E_a}^{\perp}H_1\\\nonumber&=&\nabla_{E_a}^{\perp}\nabla_{E_a}^{\perp}H_1.
\end{eqnarray}

Next, we have
\begin{equation}\label{eq:28}
\nabla^{\perp\psi}_{E_a}B^j(E_a,H_2)=\Apsi_{B^j(E_a,H_2)}E_a+\nabla^{\final}_{E_a}B^j(E_a,H_2).
\end{equation}

Since $B^j(\bar E_{\alpha},\bar E_{\alpha})=0$ and $R^{\final}(\bar E_{\alpha},H_2)\bar E_{\alpha}=-H_2$, from the Codazzi Equation of $j$, we get $(\nabla^{\perp j}_{E_a}B^j)(E_a,H_2)=0$ and then $\nabla^{\perp j}_{E_a}B^j(E_a,H_2)=B^j(\nabla^{\mathbb{C}P^p}_{E_a}E_a,H_2)$,
which leads to
$$
\nabla^{\final}_{E_a}B^j(E_a,H_2)=-A^j_{B^j(E_a,H_2)}E_a+B^j(\nabla^{\mathbb{C}P^p}_{E_a}E_a,H_2).
$$

Replacing in \eqref{eq:28}, one obtains
$$
\nabla^{\perp\psi}_{E_a}B^j(E_a,H_2)=B^j(\nabla^{\mathbb{C}P^p}_{E_a}E_a,H_2)-(A^j_{B^j(E_a,H_2)}E_a)^{\perp\psi}.
$$

From Lemma \ref{Bj2}, since $H_2=|H_2|JX$ for some unit vector field $X$ tangent to $M_2$, we get $(A^j_{B^j(E_a,H_2)}E_a)^{\perp\psi}=H_2$ and then
\begin{equation}\label{eq:36}
\nabla^{\perp\psi}_{E_a}B^j(E_a,H_2)=B^j(\nabla^{\mathbb{C}P^p}_{E_a}E_a,H_2)-H_2.
\end{equation}

Putting together \eqref{eq:26}, \eqref{eq:27} and \eqref{eq:36}, it follows
\begin{equation}\label{eq:37}
\nabla^{\perp\psi}_{E_a}\nabla^{\perp\psi}_{E_a}H^{\psi}=\frac{p}{p+q}\nabla^{\perp}_{E_a}\nabla^{\perp}_{E_a}H_1+\frac{q}{p+q}B^j(\nabla^{\mathbb{C}P^p}_{E_a}E_a,H_2)-\frac{q}{p+q}H_2.
\end{equation}
In the exact same way, we also get
\begin{equation}\label{eq:38}
\nabla^{\perp\psi}_{\bar E_{\alpha}}\nabla^{\perp\psi}_{\bar E_{\alpha}}H^{\psi}=\frac{q}{p+q}\bar\nabla^{\perp}_{\bar E_{\alpha}}\bar\nabla^{\perp}_{\bar E_{\alpha}}H_2+\frac{p}{p+q}B^j(\nabla^{\mathbb{C}P^q}_{\bar E_{\alpha}}\bar E_{\alpha},H_1)-\frac{p}{p+q}H_1.
\end{equation}

From the first equation of \eqref{eq:18}, one obtains
\begin{equation}\label{eq:39}
\nabla^{\perp\psi}_{\nabla^{\Sigma}_{E_a}E_a}H^{\psi}=\frac{p}{p+q}\nabla^{\perp}_{\nabla_{E_a}E_a}H_1+\frac{q}{p+q}B^j(\nabla_{E_a}E_a,H_2)
\end{equation}
and, from the first equation of \eqref{eq:19},
\begin{equation}\label{eq:40}
\nabla^{\perp\psi}_{\nabla^{\Sigma}_{\bar E_{\alpha}}\bar E_{\alpha}}H^{\psi}=\frac{q}{p+q}\nabla^{\perp}_{\bar\nabla_{\bar E_{\alpha}}\bar E_{\alpha}}H_2+\frac{p}{p+q}B^j(\bar\nabla_{\bar E_{\alpha}}\bar E_{\alpha},H_1).
\end{equation}

Finally, from \eqref{eq:25}, \eqref{eq:37}, \eqref{eq:38}, \eqref{eq:39} and \eqref{eq:40}, it follows that
\begin{equation}\label{eq:41}
-\Delta^{\perp\psi}H^{\psi}=-\frac{p}{p+q}\Delta^{\perp}H_1-\frac{q}{p+q}\bar\Delta^{\perp}H_2+\frac{pq}{p+q}\left(2B^j(H_1,H_2)-H_1-H_2\right).
\end{equation}

The second term in the normal part \eqref{eq:bihnormalpsi} of the biharmonic equation becomes, by the meaning of \eqref{BBB2} and the second equations of \eqref{eq:18} and \eqref{eq:19}, 
\begin{eqnarray}\label{eq:42}
\nonumber\trace B^{\psi}(\cdot,\Apsi_{H^{\psi}}\cdot)&=&\sum_{a=1}^p B^{\psi}(E_a,\Apsi_{H^{\psi}}E_a)+\sum_{\alpha=1}^q B^{\psi}(\bar E_{\alpha},\Apsi_{H^{\psi}}\bar E_{\alpha})\\&=&\frac{p}{p+q}\sum_{a=1}^p\left(B^i(E_a,A_{H_1}E_a)+B^j(E_a,A_{H_1}E_a)\right)\\\nonumber&&+\frac{q}{p+q}\sum_{\alpha=1}^q\left(B^i(\bar E_{\alpha},\bar A_{H_2}\bar E_{\alpha})+B^j(\bar E_{\alpha},\bar A_{H_2}\bar E_{\alpha})\right)\\\nonumber&=&\frac{p}{p+q}\sum_{a=1}^p B^i(E_a,A_{H_1}E_a)+\frac{q}{p+q}\sum_{\alpha=1}^q B^i(\bar E_{\alpha},\bar A_{H_2}\bar E_{\alpha})\\\nonumber&=&\frac{p}{p+q}\sum_{a=1}^p B(E_a,A_{H_1}E_a)+\frac{q}{p+q}\sum_{\alpha=1}^q \bar B(\bar E_{\alpha},\bar A_{H_2}\bar E_{\alpha})\\\nonumber&=&\frac{p}{p+q}\trace B(\cdot,A_{H_1}\cdot)+\frac{q}{p+q}\trace \bar B(\cdot,\bar A_{H_2}\cdot).
\end{eqnarray}

Replacing \eqref{H12}, \eqref{eq:41} and \eqref{eq:42} in \eqref{eq:bihnormalpsi} we see that the normal part of the biharmonic equation of $\psi$ is equivalent to the following three equations
$$
-\Delta^{\perp} H_1-\trace B(\cdot,A_{H_1}\cdot)+(p+3)H_1=0,
$$
$$
-\bar\Delta^{\perp} H_2-\trace\bar B(\cdot,\bar A_{H_2}\cdot)+(q+3)H_2=0
$$
and
$$
B^j(H_1,H_2)=0.
$$

The third equation can be written as 
$$
|B^j(H_1,H_2)|=|H_1||H_2|\left\vert B^j\left(\frac{H_1}{|H_1|},\frac{H_2}{|H_2|}\right)\right\vert=0
$$ 
and, from Lemma \ref{Bj2}, this reduces to $|H_1||H_2|=0$ which completes the proof.
\end{proof}

\begin{remark} Consider two curves $\gamma_1$ and $\gamma_2$ in $\mathbb{C}P^1$ with constant curvatures $\kappa_1$ and $\kappa_2$, respectively. This means that they are biconservative (see \cite{FLMO}) and, since any curve in $\mathbb{C}P^1$ is Lagrangian, from Theorem \ref{thm:bi2} it follows that $\psi:\Sigma^2=\gamma_1\times\gamma_2\to\mathbb{C}P^3$ is a biconservative surface, which is also CMC with $|H^{\psi}|^2=(p^2\kappa_1^2+q^2\kappa_2^2)/(p+q)^2$, but not PMC (see Remark \ref{rem:2}). 

On the other hand, a result in \cite{BCFO} shows that a CMC biconservative surface in $\mathbb{C}P^2$ with $J((JH)^{\top})$ tangent to the surface is a PMC surface. As $\Sigma^2$ satisfies all these conditions but it is not PMC, we see that the above mentioned result only holds if the surface lies in $\mathbb{C}P^2$ and also that the codimension of $\Sigma^2$ cannot be reduced (i.e., the surface does not lie in $\mathbb{C}P^2$). 

Moreover, we also know that a PMC biconservative surface in $\mathbb{C}P^n$ with $JH$ tangent to the surface lies in $\mathbb{C}P^2$ (see \cite{BCFO}). Our example shows that this only works for PMC surfaces, the CMC condition not being sufficient for this to happen.
\end{remark}

\begin{remark} There are plenty of examples of biconservative and proper-biharmonic Lagrangian submanifolds in complex projective spaces. For example, biconservative Lagrangian $H$-umbilical submanifolds are described (in a local approach) in \cite{S3}, while all proper-biharmonic Lagrangian $H$-umbilical submanifolds are determined (also locally) in \cite{S2}, after an earlier study on such surfaces was done in \cite{S1}. Also, biconservative and proper-biharmonic totally real (and in particular Lagrangian) curves in complex space forms were determined in \cite{FLMO}. Moreover, proper-biharmonic parallel Lagrangian submanifolds in $\mathbb{C}P^3$ were found in \cite{FO}. Minimal Lagrangian submanifolds in complex projective spaces were intensively studied and many characterization results as well as explicit examples were obtained (see, for example, \cite{B,CU,Chen4}). Therefore, Theorems~\ref{thm:bi1}~and~\ref{thm:bi2}, together with these results, provide two large classes of proper-biharmonic submanifolds with arbitrary dimensions and codimensions.
\end{remark}

\end{document}